\documentclass[10pt]{amsart}

\usepackage{amsmath}
\usepackage{amssymb}
\usepackage{mathrsfs}
\usepackage{enumerate}

\newcommand{\eps}{\varepsilon}

\newcommand{\Crit}{\textrm{Crit}}
\newcommand{\md}{\textrm{mod}}
\newcommand{\CV}{\textrm{CV}}

\newcommand{\tp}{\textrm{top}}

\newcommand{\dist}{\textrm{dist}}

\newcommand{\R}{\mathbb{R}}

\newcommand{\sA}{\mathscr{A}}

\newcommand{\hV}{\widehat{V}}
\newcommand{\hL}{\widehat{L}}
\newcommand{\hR}{\widehat{R}}
\newcommand{\hJ}{\widehat{J}}

\newcommand{\hW}{\widehat{W}}
\newcommand{\hU}{\widehat{U}}
\newcommand{\hB}{\widehat{B}}

\newtheorem{theorem}{Theorem} 
\newtheorem{lemma}{Lemma}     

\newtheorem{proposition}{Proposition}

\title[]
{Interval maps quasi-symmetrically conjugated to a  piecewise affine map}

\author{Huaibin Li}
\thanks{The author was partially supported by
the National Natural Science Foundation of China (Grant No. 11101124) and
the FONDECYT grant 3110060 of Chile.}
\address{Huaibin Li, Facultad de Matem{\'a}ticas, Pontificia Universidad Cat{\'o}lica de Chile, Avenida Vicu{\~n}a Mackenna~4860, Santiago, Chile}
\subjclass[2010]{37E05}
\keywords{Interval maps, maximal entropy measure, quasi-symmetry conjugacy, semi-hyperbolicity}

\begin{document}
\maketitle

\begin{abstract}
Consider a  multimodal interval map $f$ of $C^3$ with non-flat critical points.  We establish  several characterizations of the map~$f$ is  quasi-symmetrically conjugated to a  piecewise affine map in the case $f$ is topologically exact and all of its periodic points are hyperbolic repelling. In particular, we give a negative answer to a question posted by Henk Bruin in~\cite{Bru96}.
\end{abstract}

\section{Introduction}\label{sec:intro}

Multimodal interval maps are often topologically conjugate to piecewise linear maps. So in a topological sense,
such a multimodal map is the same as a piecewise linear map. In a metric sense, however, the difference is
clear. Some metric similarities still occur, when the conjugacy satisfies certain constraints.

Let $I$ be a compact interval of $\R$. A homeomorphism $h: I \to I$ on the
interval is \emph{quasi-symmetric} if there exists constant $K \geq 1$ such that for all $x\in I$ and
all $\eps> 0$ with $x\pm\eps$ in $I$ we have $$\frac{1}{K}\leq \frac{|h(x+\eps)-h(x)|}{|h(x)-h(x+\eps)|}\leq K.$$
The notion of quasi-symmetry gained importance since Sullivan~\cite{Sul86} proved the
following rigidity  result: Let $f_a(x)=1-ax^2$, if $f_a$ and $f_b$ are quasi-symmetrically conjugate
and do not have a periodic attractor, then $a= b.$

In this paper, we give several characterizations of a multimodal interval map that is quasi-symmetrically conjugate to a piecewise affine map. To state our main results, let us be more precise. Let $I$ be a compact interval of $\R.$ A non-injective continuous map $f: I\to I$ is \emph{multimodal}, if there is a finite
partition of $I$ into intervals on each of which $f$ is injective. A multimodal map $f: I\to I$ is \emph{topologically exact}, if for every open subset~$U$ of~$I$ there is an integer $n\geq 1$ such
that $f^n(U) = I.$
A \emph{turning point} of a multimodal map $f: I \to I$ is a point in $I$ at which $f$ is not locally injective.
For a differentiable multimodal map $f: I\to I$,  a point of~$I$ is \emph{critical of~$f$} if the derivative of~$f$ vanishes at it. In what follows, denote by~$\Crit(f)$ the set of critical points of~$f$, denote by $\Crit'(f)$ the turning points of $f$,  and put $\CV(f):=f(\Crit(f))$.  A $C^1$ multimodal map~$f : I \to I$ is \emph{of class~$C^3$ with non-flat critical points}, if
\begin{itemize}
\item
The map~$f$ is of class~$C^3$ outside $\Crit(f)$;
\item
For each critical point~$c$ of~$f$ there exists a number $\ell_c>1$ and diffeomorphisms~$\phi$ and~$\psi$ of~$\R$ of class~$C^3$, such that $\phi(c)=\psi(f(c))=0$, and such that on a neighborhood of~$c$ on~$I$, we have
$|\psi\circ f| = |\phi|^{\ell_c}. $
\end{itemize}
Recall that for an integer $n\geq 1$, a periodic
point $p$ of $f$ of period $n$ is \emph{hyperbolic repelling} if $|Df^n(p)|>1,$  and that  a critical point $c\in \Crit(f)$ is called \emph{recurrent} if $c\in \omega(c),$ where $\omega(c)$ denote the $\omega$-limit set of $c$ that is the set of  accumulation  points of the forward orbit $\{f^n(c)\}_{n=0}^{+\infty}$ of $c$.

The \emph{topological entropy $h_\tp(f)$ of $f$} is equal to
the supremum of the metric entropies $h_\mu(f)$ taken over all $f$-invariant Borel probability
measures $\mu$, see for example~\cite{Wal82}.
A $f$-invariant Borel probability measure $\mu$ such that $h_\mu(f) = h_\tp(f)$ is called a \emph{maximal entropy measure.}
It is well-known that a multimodal   interval map $f:I \to I$ that is topologically exact has a unique maximal entropy measure $\mu_f$. Moreover, the topological support of $\mu_f$ is equal to $I$, and the Jacobian of~$\mu_f$ is constant equal to $\exp(h_{\tp}(f)).$

For a point $x$ in $I$, $r>0$, an integer $m\geq 1,$ and each $j$ in $\{0, 1, \cdots, m-1\}$, let $W_j$ be the pull-back of $B(f^m(x), r)\cap I$ by $f^{m-j}$ containing $f^j(x)$.  The \emph{criticality of $f^m$ at $x$ with respect to $r$} is defined as the following number $$\# \{j\in \{0, 1, \cdots, m-1\}: W_j\cap \Crit'(f)\neq \emptyset\}.$$ Moreover, the map $f$ is said to be \emph{semi-hyperbolic}, if there exist constants $r>0$ and $D\geq 1$ such that for every $x$ in $I$ and each integer $n\geq 1$ the criticality of $f^n$ at $x$ with respect to $r$ is at most $D$.

The main result of this paper is the following theorem.
\begin{theorem}\label{thm:qs}
Let $f:I \to I$ be a multimodal map of class~$C^3$ with non-flat critical points and with all periodic points hyperbolic repelling. If $f$ is topologically exact, then the following statements are equivalent.
\begin{enumerate}[(1).]
\item $f$ is semi-hyperbolic;
\item  $f$ has no recurrent critical points;
\item The maximal entropy measure of $f$ is doubling;
\item $f$ is quasi-symmetrically conjugate to a piecewise affine  function with slope equal to $\pm \exp (h_\tp(f)).$
\end{enumerate}
\end{theorem}
Recall that a Borel measure $\mu$  on a metric space  $(X, \dist)$ is said to be \emph{doubling}, if there are constants $C_*>0$ and $r_*>0$ such that for each $x$ in $X$ and $r$ in $(0, r_*)$ we have $$\mu(B(x, 2r))\leq C_* \mu(B(x,r)).$$

The concept of semi-hyperbolicity was introduced by Carleson, Jones and Yoccoz to characterize those complex polynomials whose basin of attraction of infinity is a John domain~\cite{CarJonYoc94}. In the context of complex polynomials, they proved the equivalence between $(1)$ and $(2)$ of Theorem~\ref{thm:qs}. See also~\cite{Yin99, Mih11} for an extension to complex rational maps. The proof in the case of interval maps is simpler (Lemma~\ref{l:sei2nonre} in~\S\ref{1and2}), thanks to the backward Lyapunov stability of interval maps (Lemma~\ref{l:pullstable} in~\S\ref{1and2}).
The equivalence between conditions $(3)$ and $(4)$ of Theorem~\ref{thm:qs} is a simple consequence of the theory of Parry and Milnor and Thurston~\cite{Par66, MilThu88}, see Proposition~\ref{p:qs2doubling} in \S\ref{3and4}. The implication $(2)\Rightarrow (4)$ is new even for unimodal map. It gives an alternative proof of~\cite[Theorem~1]{Bru96} when combined with~\cite{Sch11}. The implication $(4)\Rightarrow (2)$ is also new, when restricted to unimodal maps it answers a question posted by Bruin~\cite[Question]{Bru96}, see also~\cite{BruBru04}: \emph{Are there non-Misiurewicz maps that are quasi-symmetrically conjugate to tent-maps?} Recall that a unimodal interval map $f$ satisfies the \emph{Misiurewicz condition} if the critical point of $f$ is not periodic, and has a forward orbit which stays away from itself. We notice that for unimodal maps the Misiurewicz condition coincides with semi-hyperbolicity, but these conditions are different for maps with several critical points.

The core of the proof of Theorem~\ref{thm:qs} is the proof of the implications $(1)\Rightarrow (3)$ (Proposition~\ref{semi2doubl} in~\S\ref{123}) and $(3)\Rightarrow (2)$ (Proposition~\ref{semidoubl} in~\S\ref{322}). The proof of these implications follow in general lines used in the proof of the correspondent implications of~\cite[Theorem~A]{Riv10}, and of the existence of nice couples with arbitrarily large modulus for TCE maps, see~\cite{PrzRiv07}. However, the proofs are different of various places, mainly due to the fact that interval maps are not open in general.

\begin{proof}[Proof of Theorem~\ref{thm:qs}]
The equivalence between $(1)$ and $(2)$ is shown in Lemma~\ref{l:sei2nonre}. That $(1)$ implies $(3)$ follows from Proposition~\ref{semi2doubl}.   That $(3)$ implies $(2)$ is Proposition~\ref{semidoubl}. The equivalence between $(3)$ and $(4)$ is shown in Proposition~\ref{p:qs2doubling}.
\end{proof}

\subsection{Notations}
Throughout the rest of this paper, let $I=[0,1]$ endowed with the distance induced by the absolute value $|\cdot|$ on $\R,$ and denote by~$\sA$ the collection of interval maps of class~$C^3$ with non-flat critical points and with all periodic points hyperbolic repelling. Put $$\ell_{\max}:=\max\{\ell_c: c\in \Crit(f)\} \text{  and  }  \ell_{\min}:=\min\{\ell_c: c\in \Crit(f)\}.$$  For an interval $J$
contained in $I,$ we denote by $|J|$ its length and for $\eta > 0$ we denote by~$\eta J$ the open
interval of~$\R$ of length $(1+2\eta) |J|$ that has the same middle point as~$J.$

For each $x\in I$ and $r>0$ we denote by $B_I(x, r)$ the ball of $I$ centered at $x$ and of radius $r$ defined as follows: $B_I(x,r):=\{y\in I: |x-y|<r\}.$

\subsection{Acknowledgments}
\label{ss:Acknowledgements}
The main idea of this paper came to the author after several discussions with Juan Rivera-Letelier on his work~\cite{Riv10}, and he read an earlier version very carefully. I am very grateful to him for those stimulating conversations  and for his useful comments and corrections to an earlier version of this paper.

\section{Non-recurrence and quasi-symmetry}
In this section, we prove the equivalences between $(1)$ and $(2)$, $(3)$ and $(4)$ in Theorem~\ref{thm:qs}.
\subsection{Semi-hyperbolicity and non-recurrent critical points}\label{1and2}
Given a multmodal map $f: I\to I,$ an integer $n\geq 1$ and a subset $J$ of $I$, a connected component of $f^{-n}(J)$ will be called a \emph{pull-back} of $J$ by $f^n.$ The following general facts of multimodal maps is used for several time in what follows, see for example~\cite{Riv1206} for a proof.
\begin{lemma}[Lemma~A.2, \cite{Riv1206}]\label{l:pullstable}
Let $f: I \to I$ be an interval map in~$\sA$ that is topologically exact.
Then for every $\kappa> 0$ there is $\delta> 0$ such that for every $x$ in $I$, every integer $n\geq 1$, and every pull-back $W$ of $B(x, \delta)$ by $f^n$, we have $|W| < \kappa.$
\end{lemma}

In the complex setting, the following result was shown by Carleson, Jones and Yoccoz in~\cite{CarJonYoc94}, see also~\cite{Yin99, Mih11}. Thanks to Lemma~\ref{l:pullstable}, the proof in the real case is much simpler than that in the complex setting.
\begin{lemma}\label{l:sei2nonre}
$f$ is semi-hyperbolic if and only if it has no recurrent critical points.
\end{lemma}
\begin{proof}
First we assume that $f$ has no recurrent critical points. Then there exists $\delta_1>0$ such that for each $c\in \Crit(f)$ and each integer $n\geq 1$ we have $f^n(c) \not\in B(c, \delta_1).$ Reducing $\delta_1$ if necessary we assume that $$\inf\{|c_1-c_2|: c_1, c_2\in \Crit(f), c_1\neq c_2\}>\delta_1.$$  Let $\delta$ be the constant given by Lemma~\ref{l:pullstable} with $\kappa=\delta_1.$ It follows that for each $x$ in~$I$, each integer $m\geq 1$ and each $c\in \Crit(f)$, there is at most one~$j$ in $\{0, \cdots, m-1\}$ such that the pull-back $B(f^m(x), \delta)$ by $f^{m-j}$ containing $f^j(x)$ contains $c,$ since otherwise there is $k\geq 1$ such that $f^k(c)\in B(c, \delta_1).$ Hence, the criticality of~$f^n$ at~$x$ with respect to $\delta$ is at most $\#\Crit(f)$. This implies~$f$ is semi-hyperbolic.

Now let us assume $f$ is semi-hyperbolic. Then there exist constants $r>0$ and $D\geq 1$ such that for every $x$ in $I$ and each integer $n\geq 1$ the criticality of $f^n$ at $x$ with respect to $r$ is at most $D$.  Arguing by contradiction, assume there is a recurrent critical point $c.$  Then we can define inductively a sequence $(n_i)_{i=1}^{+\infty}$ of positive integers and a sequence $\{W_i\}_{i=0}^{+\infty}$ of open intervals such that $W_0=B(c, r/2)$, and for each $i\geq 1$ we have $c\in W_i$, $f^{n_i}(c)\in W_{i-1}$ and $W_i$ is  the pull-back of $W_{i-1}$ by $f^{n_i}$ containing $c.$  Put $S:=n_1+\cdots+n_{D+1}$. For each~$x\in W_{D+1}$, we have $B(c, r/2)\subset B(f^S(x), r)$. It follows that the criticality of~$f^S$ at~$x$ with respect to $r$ is at least $D+1$. We obtain a contradiction, and so~$f$ has no recurrent critical points.
\end{proof}

\subsection{Maximal entropy measure and the piecewise affine modal}\label{3and4}

The following lemma is a direct consequence of the compactness arguments, we omit the proof.
\begin{lemma}\label{positiveopenset}
 Let $\mu$ be a non-atomic Borel measure whose support is equal to $I.$ Then for each $\delta>0$ we have $\inf_{x\in I}\mu(B_I(x, \delta))>0.$
\end{lemma}


 \begin{lemma}\label{measureadj}
 Let $\mu$ be a non-atomic Borel measure whose support is equal to $I.$ Then~$\mu$ is doubling if and only if there exists $C>1$ such that for each $x$ in~$I$ and each $\eps>0$ with $x \pm \eps$ in~$I,$ we have $$C^{-1}\mu((x, x+\eps))\leq \mu((x-\eps, x))\leq C\mu((x, x+\eps)).$$
\end{lemma}
\begin{proof}
Assume that $\mu$ is doubling with constants $C_*>1$ and $r_*>0$. For each $\eps>0$ and each $x$ in~$I$ with $x\pm \eps$ in $I$, we have $$(x, x+\eps)\subset B(x-\frac{\eps}{2}, 2\eps) \text{ and } (x-\eps, x)\subset B(x+\frac{\eps}{2}, 2\eps).$$ To prove the desired inequality, we consider the following two cases:

(i) If $\eps$ is in $(0,r_*].$  Then by the doubling property of $\mu$ we obtain $$\mu((x, x+\eps))\leq \mu(B_I(x-\frac{\eps}{2}, 2\eps))\leq C_*^2 \mu(B_I(x-\frac{\eps}{2}, \frac{\eps}{2}))=C_*^2 \mu((x-\eps, x)),$$
$$\mu((x-\eps, x))\leq \mu(B_I(x+\frac{\eps}{2}, 2\eps))\leq C_*^2 \mu(B_I(x+\frac{\eps}{2}, \frac{\eps}{2}))=C_*^2 \mu((x, x+\eps)).$$
This proves the desired inequality with $C_1:=C_*^2$.

(ii) If $\eps>r_*,$  put $$C_+:=\inf\{\mu(J): J\subset I \text{ is an interval with } |J|\geq r_*\}.$$ Then by Lemma~\ref{positiveopenset} we have $C_+>0.$ It follows that $$\mu((x, x+\eps))\leq 1\leq C_+^{-1}\mu((x-r_*, x))\leq  C_+^{-1}\mu((x-\eps, x))$$ and $$\mu((x-\eps, x))\leq 1\leq C_+^{-1}\mu((x, x+r_*))\leq  C_+^{-1}\mu((x, x+\eps)).$$
This proves again the desired inequality with $C:=C_+^{-1}$.

To prove the converse statement, fix $x$ in $I$ and $r>0.$ First we prove that
\begin{equation}\label{eqb}
\mu((x-2r, x-r)\cap I)\leq C\mu((x-r, x)\cap I).
\end{equation}
In fact, if $(x-2r, x-r)\cap I=\emptyset,$ then the inequality~(\ref{eqb}) is trivial. Otherwise,  put $s:=|(x-2r, x-r)\cap I|.$
Then $s\leq r$ and by the assumption with $x$ replaced by $x-r$ and $\eps$ replaced by $s,$ we have $$\mu((x-2r, x-r)\cap I)=\mu((x-r-s, x-r))\leq C\mu((x-r, x-r+s))\leq C\mu((x-r, x)\cap I).$$ The same argument gives us $$\mu((x, x+r)\cap I)\geq C^{-1}\mu((x+r, x+2r)\cap I).$$ Therefore, we have
\begin{multline*}
\mu(B_I(x,2r))=\mu(B(x,2r)\cap I)\\=\mu((x-2r, x-r)\cap I)+\mu([x-r, x+r]\cap I)+\mu((x+r, x+2r)\cap I)\\ \leq C\mu((x-r, x)\cap I) +\mu([x-r, x+r]\cap I) +C \mu((x, x+r)\cap I)\\\leq (1+C)\mu(B(x, r)\cap I)=(1+C)\mu(B_I(x, r)),
\end{multline*}
and proves $\mu$ is doubling with $C_*=1+C.$
\end{proof}

As noted before, the following well-known lemma is used several times throughout the rest of this article, see for example~\cite{Hof79,Hof81, MilThu88}.
\begin{lemma}\label{l:maxmeasure}
Let $f:I \to I$ be a multimodal   interval map that is topologically exact. Then there is a unique maximal entropy measure $\mu_f$ of $f$. Moreover, the topological support of $\mu_f$ is equal to $I$, $\mu_f$ is non-atomic and the Jacobian of $\mu_f$ is constant equal to $\exp(h_{\tp}(f)).$
\end{lemma}

\begin{proposition}\label{p:qs2doubling}
Let $f:I \to I$ be a multimodal map in $\sA$ that is topologically exact, and let $\mu_f$ be the maximal entropy measure given by Lemma~\ref{l:maxmeasure}. Then
The following statements are equivalent.
\begin{enumerate}[1.]
\item $f$ is quasi-symmetrically conjugate to a piecewise affine map from the interval $I=[0,1]$ to itself with slope equal to $\pm \exp (h_\tp(f))$ everywhere;
\item  The measure $\mu_f$ is doubling.
\end{enumerate}
\end{proposition}
Recall that $\mu_f$ is doubling if there is $C_*>1$ and $r_*>0$ such that for each $x$ in~$I$ and each~$r$ in $(0, r_*]$ we have $$\mu_f(B_I(x, 2r))\leq C_* \mu_f(B_I(x, r)).$$
\begin{proof}[Proof of Proposition~\ref{p:qs2doubling}] Put $s:=\exp(h_\tp(f))$ and for each $x$ in $I$, put $h(x):= \mu_f([0, x]).$ Note that~$\mu_f$ is non-atomic and the support of~$\mu_f$ is equal to $I.$ Then we have that $h: I \to I$ is continuous bijective function, so a homeomorphism. Defining  $F:=h\circ f\circ h^{-1},$  we have $h\circ f= F\circ h,$ and $F$ maps $I$ to itself. Furthermore, a completely analogous arguments as that in the proof of~\cite[Theorem~7.4]{MilThu88} shows that the map $F:I \to I$ is piecewise affine with slope equal to $\pm s$ everywhere. Conversely, suppose there is an increasing homeomorphism $\widetilde{h}: I\to I$ and a piecewise linear function $\widetilde{F}:I \to I$ with slope equal to $\pm s$ everywhere such that $\widetilde{h}\circ f= \widetilde{F}\circ \widetilde{h}.$ Denote by $\text{Leb}$ the Lebegue  measure on~$I$ and put~$\mu:=(\widetilde{h}^{-1})_* \text{Leb}.$ Then we have $\mu(I)=1,$ and for every Borel set $A$ of $I$ on which $f$ is injective we have \begin{multline*}
\mu(f(A))=(\widetilde{h}^{-1})_*\text{Leb}(f(A))=\text{Leb}(\widetilde{h}(f(A)))=\text{Leb}(\widetilde{F}(\widetilde{h}(A)))\\=s |\widetilde{h}(A)|=s\text{Leb}(\widetilde{h}(A))=s (\widetilde{h}^{-1})_*\text{Leb}(A)=s\mu(A).
\end{multline*}
It follows that the Jacobian of $\mu$ satisfies~$\text{Jac} (\mu)= s=  \exp(h_\tp(f)).$
 By Rokhlin's entropy
inequality, see~\cite[\S 2.9]{PrzUrb10} or~\cite[\S 10]{Par69} for example, we have
$$h_\mu(f)\geq \int_I \log \text{Jac}(\mu)\ d\mu =\int_I h_\tp(f) \ d\mu=h_\tp(f).$$ This gives us $\mu$ is a maximal entropy measure of $f$. Hence, by Lemma~\ref{l:maxmeasure}, we have~$\mu=\mu_f.$ Furthermore,  by the definition of $\mu$ we have $$\mu([0,x])=\text{Leb}(\widetilde{h}([0,x]))=|\widetilde{h}(0)-\widetilde{h}(x)|.$$  It follows that  $$\widetilde{h}(x)=\mu([0,x])=\mu_f([0,x])=h(x)$$
Therefore, to end the proof, it is enough to prove that $h$ is quasi-symmetric if and only if $\mu_f$ is doubling.

By the definition we know that $h$ is quasi-symmetric if and only if there is $M>1$ such that for each $x\in I $ and $\eps>0$ with $x\pm \eps$ in $I$ we have $$M^{-1}\leq \frac{|h(x+\eps)-h(x)|}{|h(x)-h(x-\eps)|}\leq M.$$ This, by the definition of $h$, is equivalent to the following
\begin{equation}\label{e:doub}
M^{-1}\leq \frac{|\mu_f([0, x+\eps])-\mu_f([0, x])|}{|\mu_f([0,x])-\mu_f([0, x-\eps])|} =\frac{\mu_f((x, x+\eps])}{\mu_f((x-\eps, x])}\leq M.
 \end{equation}
 Note that $\mu_f$ is non-atomic. Then by Lemma~\ref{measureadj} we have that the equality~(\ref{e:doub}) above holds if and only if $\mu_f$ is doubling, and complete the proof.
\end{proof}

\section{Semi-hyperbolicity and doubling}\label{123} 
The main goal of this section is prove the following proposition which gives the implication $(1)\Rightarrow (3)$ of Theorem~\ref{thm:qs}.
\begin{proposition}\label{semi2doubl}
Let $f:I \to I$ be a multimodal map in $\sA$ that is topologically exact, and let $\mu_f$ be the maximal entropy measure given by Lemma~\ref{l:maxmeasure}.
Assume that $f$ is semi-hyperbolic, then the measure $\mu_f$ is doubling.
\end{proposition}
The proof of this proposition, depending on several lemmas, is given at the end of this section.

Let~$\tau > 0$ and let $J_1 \subset J_2$ be two intervals of $I$. We say that $J_1$ is \emph{$\tau$-well inside~$J_2$,} if both components of $J_2\setminus J_1$ have length at least~$\tau|J|$.

\begin{lemma}[Theorem~A, \cite{LiShe10a}]\label{koebe}
 Let $f:I \to I$ be a multimodal map in $\sA$. Then for each $\tau > 0$ there exist
$C>1$ and $\xi>0$ satisfying the following. Let $T\subset I$ be an open
interval, $J$ a closed subinterval of $T$  and an integer $s\geq 1$ such that the following hold:
\begin{enumerate}[1.]
\item $f^s: T \to f^s(T)$ is a diffeomorphism;
\item $|f^s(T)|\leq \xi$;
\item $f^s(J)$ is $\tau$-well inside $f^s(T).$
\end{enumerate}
Then for each pair $x$ and $y$ of points in $J$ we have
$$\frac{|Df^s(x)|}{|Df^s(y)|}\leq C.$$
Furthermore, $C \to 1$  as $\tau \to +\infty.$
\end{lemma}

A \emph{chain} is a sequence of open intervals $\{G_i\}^s_{i=0}$ such that for each $0\leq i < s$, $G_i$ is a
component of $f^{-1}(G_{i+1})$. The \emph{criticality} of the chain is the number of $i'$s such that $G_i$ contains a
critical point.

The following is a version of the Koebe principle for non-diffeomophic pull-backs, see for example~\cite{LiShe08b} for a proof.
\begin{lemma}[Lemma~4.1, \cite{LiShe08b}] \label{improvedkoebe}
Let $f:I \to I$ be a multimodal map in $\sA$, then there exists $\eta > 0$ such that for each $\tau>0$ and $N\geq 1$ there are $\tau'>0$ and $C>0$ with
the following property. Let $\{H_j \}^s_{j=0}$ and $\{H'_j \}^s_{j=0}$ be chains such that $H'_j\subset H_j$ for all $0\leq  j\leq s.$
Assume that $|H_s|<\eta,$ $H'_s$ is $\tau$-well inside $H_s$ and that the criticality  of the chain $\{H_j\}^s_{j=0}$ is at most $N.$ Then the following holds:
\begin{enumerate}[1.]
\item $H'_0$ is $\tau'$-well inside in $H_0;$
\item for each $x\in H'_0$, $$|Df^s(x)|\leq C \frac{H'_s}{H'_0},$$
\end{enumerate}
Moreover, for a fixed $N,$
$\log \tau'
/ \log \tau$ tends to a positive constant as $\tau \to +\infty.$
\end{lemma}

\begin{lemma}\label{impkoebe}
Let $f:I \to I$ be a multimodal map in $\sA$. Then for each $N\geq 1$ and $\tau>0$ there exist $\delta'>0$ and $C_1>1$ such that the following holds. Let $J_0\subset \hJ_0$ be two subintervals of $I$, and let~$m\geq 1$ be an integer. If $(1+2\tau)|f^m(\hJ_0)|<\delta', \tau f^m(\hJ_0)\subset I$ and the number of $i'$s such that the pull-back of $\tau f^m(\hJ_0)$ by $f^{m-i}$ containing $f^i(\hJ_0)$ intersects $\Crit(f)$ is at most $N,$ then we have $$C_1^{-1}\left(\frac{|f^m(\hJ_0)|}{|f^m(J_0)|}\right)^{\ell_{\max}^{-N}}\leq \frac{|\hJ_0|}{|J_0|}\leq C_1 \frac{|f^m(\hJ_0)|}{|f^m(J_0)|}.$$
\end{lemma}
\begin{proof}
By the non-flatness of critical points, there are $\kappa>0$ and $M>0$ such that for every $c\in \Crit(f)$ and every interval $J$ contained in $B(c, \kappa)$ we have
\begin{equation}\label{nonflat}
 M^{-1}|J|^{\ell_c}\leq |f(J)|\leq M|J|^{\ell_c}
\end{equation}
Reducing $\kappa$ if necessary, we can assume that for any two distinct critical $c$ and $c'$ we have $|c-c'|>\kappa.$ By Lemma~\ref{l:pullstable} there is $\delta> 0$ such that for every interval $J$ of $I$ with $|J|<\delta$, every integer $n\geq 1$, every pull-back $W$ of $J$ by $f^n$ has length at most $\kappa.$  Let $\eta>0$, $\tau'>0$ and $C>0$ be the constants given by Lemma~\ref{improvedkoebe} with~$N$ and~$\tau$, and let $C'>1$ and $\xi>0$ be the constants given by Lemma~\ref{koebe} for $\tau'.$ Put $\delta':=\min\{\kappa, \delta, \xi, \eta\}.$ To prove the lemma, let $J_0\subset \hJ_0$ be two subintervals of $I$ and let $m\geq 1$  be an integer such that $(1+2\tau)|f^m(\hJ_0)|<\delta', \tau f^m(\hJ_0)\subset I$ and the number of $i'$s such that the pull-back of $\tau f^m(\hJ_0)$ by $f^{m-i}$ containing $f^i(\hJ_0)$ intersects $\Crit(f)$ is at most $N.$ Moreover,  let $\{H_i\}_{i=0}^{m}$ be the chain such that $H_m=\tau f^m(\hJ_0)$ and $H_0$ is the pull-back of $\tau f^m(J_2)$ by $f^m$ containing $J_2$. Let $0\leq n_1<n_2<\cdots<n_s<n_{s+1}=m$ be all the integers $i$ with $H_i\cap \Crit(f)\neq \emptyset.$ Then $s\leq N$, and by Lemma~\ref{improvedkoebe} for every $i$ in $\{1, \cdots, s+1\}$ we have $f^{n_i}(\hJ_0)$ is $\tau'$-well inside in $H_{n_i}.$ Therefore, by Lemma~\ref{koebe} for every $i$ in $\{1, \cdots, s\}$ we have $$C'\frac{|f^{n_{i+1}}(\hJ_0)|}{|f^{n_{i+1}}(J_0)|}\geq\frac{|f^{n_{i}+1}(\hJ_0)|}{|f^{n_{i}+1}(J_0)|}\geq \frac{1}{C'}\frac{|f^{n_{i+1}}(\hJ_0)|}{|f^{n_{i+1}}(J_0)|}$$  and  $$ C'\frac{|f^{n_{1}}(\hJ_0)|}{|f^{n_{1}}(J_0)|}\geq \frac{|\hJ_0|}{|J_0|}\geq \frac{1}{C'}\frac{|f^{n_{1}}(\hJ_0)|}{|f^{n_{1}}(J_0)|}.$$ On the other hand, by~(\ref{nonflat}) for every $i$ in $\{1, \cdots, s\}$ we have $$M^{2/\ell_{\min}}\frac{|f^{n_i+1}(\hJ_0)|}{|f^{n_i+1}(J_0)|} \geq \frac{|f^{n_i}(\hJ_0)|}{|f^{n_i}(J_0)|} \geq M^{-2/\ell_{\min}}\left(\frac{|f^{n_i+1}(\hJ_0)|}{|f^{n_i+1}(J_0)|}\right)^{1/\ell_{\max}}.$$ It follows that $$C'^{-(1+N)}M^{-\frac{2N}{\ell_{\min}}} \left(\frac{|f^m(\hJ_0)|}{|f^m(J_0)|}\right)^{\ell_{\max}^{-N}} \leq \frac{|\hJ_0|}{|J_0|}\leq C'^{(1+N)}M^{\frac{2N}{\ell_{\min}}} \frac{|f^m(\hJ_0)|}{|f^m(J_0)|}.$$ This proves the lemma with $C_1:=C'^{(1+N)}M^{\frac{2N}{\ell_{\min}}}.$
\end{proof}

\begin{proof}[Proof of Proposition~\ref{semi2doubl}]
Since $f$ is semi-hyperbolic, then there exist constants $r'>0$  and $D\geq 1$ such that for every $x$ in $I$ and each integer $n\geq 1$ the criticality of $f^n$ at $x$ with respect to $r'$ is at most $D$. Let $C_1>1$ and $\tau'>0$ be the constants given by Lemma~\ref{impkoebe} with $\tau=2$ and $N=D$. Put $M:=\sup_{I}|Df|$ and $r_0:=\min\{r', \tau'\}.$ Fix $x\in I$ and a sufficiently small $r>0$. Let $m$ be the minimal integer such that $|f^m(B_I(x, 2r))|\geq r_0/(6M).$ Since $f$ is topologically exact, such integer $m$ exists. Note that $|f^{m-1}(B_I(x, 2r))|<r_0/(6M),$ so $|f^m(B_I(x, 2r))|\leq r_0/6.$ By our choice of $r_0$ we know that the number of $i'$s such that the pull-back of $2 f^m(B_I(x,2r))$ by $f^{m-i}$ containing $f^i(B_I(x,2r))$ intersects $\Crit(f)$ is at most $D,$ and $$|(1+4)|f^m(B_I(x,2r))|<\tau'.$$ Therefore, by Lemma~\ref{impkoebe} we have  $$2\geq\frac{|B_I(x,2r)|}{|B_I(x,r)|}\geq C_1^{-1}\left(\frac{|f^m(B_I(x, 2r))|}{|f^m(B_I(x,r))|}\right)^{\ell_{\max}^{-D}}.$$ This implies $|f^m(B_I(x,r))|\geq (2C_1)^{\ell_{\max}^{-D}}(6M)^{-1} r_0.$  Moreover,  by Lemma~\ref{positiveopenset},  we have $$\delta':=\inf_{x\in I} \mu_f(B_I(x, (2C_1)^{\ell_{\max}^{-D}}(12M)^{-1} r_0 ))>0.$$
On the other hand, by Lemma~\ref{l:pul} we have $$\mu_f(B_I(x, r))\geq \exp(-m h_\tp(f))\mu_f(f^m(B_I(x,r)))$$ and $$\mu_f(B_I(x, 2r))\leq 2^D\exp(-m h_\tp(f))\mu_f(f^m(B_I(x, 2r))).$$ It follows that $$\frac{\mu_f(B_I(x, 2r))}{\mu_f(B_I(x, r))}\leq 2^D\frac{\mu_f(f^m(B_I(x, 2r)))}{\mu_f(f^m(B_I(x,r)))}.$$ Hence,  $$\mu_f(B_I(x, 2r))\leq 2^D\frac{\mu_f(f^m(B_I(x, 2r)))}{\mu_f(f^m(B_I(x,r)))}\mu_f(B_I(x,r))\leq \frac{2^D}{\delta'}\mu_f(B_I(x,r)).$$ This implies $\mu_f$ is doubling.
\end{proof}

\section{Doubling implies non-recurrent critical points}\label{322}
In this section, our main goal is to prove the following proposition. The proof,
which is given at the end of this section, depends on several lemmas.

\begin{proposition}\label{semidoubl}
Let $f:I \to I$ be a multimodal map in $\sA$ that is topologically exact, and let $\mu_f$ be the maximal entropy measure given by Lemma~\ref{l:maxmeasure}. Assume the  measure $\mu_f$ of $f$ is doubling, then $f$ has no recurrent critical points.
\end{proposition}

We start with the following observation, see for example~\cite{Riv10} for a proof.
\begin{lemma}[Lemma~1, \cite{Riv10}]\label{l:doubleboundedfrombelow}
Let $(X, \dist)$ be a compact metric space and let $\mu$ be a doubling measure
on $X.$ Then there are constants $C > 0$ and $\alpha > 0$ such that for each sufficiently
small $r > 0$ and each $x\in X$ we have
$$\mu(B(x, r))\geq C r^\alpha.$$
\end{lemma}

Throughout the rest of this section, fix a multimodal interval map $f: I\to I$ in~$\sA$ that is topologically exact, and let $\mu_f$ be its maximal entropy measure given by Lemma~\ref{l:maxmeasure}. In particular, $\mu_f$ is non-atomic and its support is equal to $I$.

\subsection{preliminaries}
\begin{lemma}\label{lebsmall2doublsmall}
If $\mu_f$ is doubling, then for each $M'>1$ there is $\eps>0$ such that the following holds. For each pair of adjacent subintervals $L$ and $R$ of $I$ with $|L|\geq \eps |R|$, we have $\mu_f(L)\geq M'\mu_f(R).$
\end{lemma}
\begin{proof}
Let  $C>1$  be the constants given by Lemma~\ref{measureadj}. For any $M'>1$, let $m$ be the minimal integer such that $C^{-1}(1+C^{-1})^m\geq M',$ and put $\eps:=2^m.$ For each pair of adjacent subintervals $L$ and $R$ of $I$ with $|L|\geq \eps |R|$, denote by~$a$ the common endpoint of~$L$ and~$R$, and for every $i\in \{0, 1, \cdots, m\}$ put $L_i:=L\cap B(a, 2^i|R|).$  By Lemma~\ref{measureadj} we know $\mu_f(L_0)\geq C^{-1}\mu_f(R)$, and for each every $i\in \{0, 1, \cdots, m-1\}$ we have $$\mu_f(L_{i+1})\geq (1+C^{-1})\mu_f(L_i).$$ It follows inductively that $$\mu_f(L)\geq \mu_f(L_m)\geq (1+C^{-1})^m\mu_f(L_0)\geq C^{-1}(1+C^{-1})^m \mu_f(R)\geq M'\mu_f(R).$$ This completes the proof.
\end{proof}

\begin{lemma}\label{l:pul}
There exist $\tau>0$ such that the following holds. For each open subinterval $V$ of $I$ with $|V|< \tau$, every integer $m\geq 1$ and each pull-back $W$ of $V$ by $f^m$, if we denote by $D$ the the number of those $j\in \{0, \cdots, m -1\}$ such that the connected component of $f^{-(m-j)}(V)$ containing $f^j(W)$ intersects $\Crit(f),$ then \begin{multline*}
\exp(-m h_\tp(f))\mu_f(f^m(W)) \leq \mu_f(W)\leq 2^D \exp(-m h_\tp(f)) \mu_f(f^m(W))\\\leq 2^D \exp(-m h_\tp(f)) \mu_f(V).
\end{multline*}
\end{lemma}
\begin{proof}
Let $\kappa'>0$ be such that for each pair of critical points $c_1\neq c_2$ in $\Crit(f)$ we have $|c_1-c_2|>\kappa',$ and $\tau$ the constant given by Lemma~\ref{l:pullstable} with $\kappa=\kappa'$. We will prove the assertion holds for such $\tau.$ In fact, fix an open subinterval~$V$ of~$I$ with $|V|<\tau,$ an integer $m\geq 1$ and a pull-back $W$ of $V$ by $f^m$. For each $j$ in $\{0, \cdots, m\}$, let $\hW_j$ be the connected component of $f^{-(m-j)}(V)$ containing $f^{j}(W).$ Let $0\leq n_1<n_2< \cdots <n_D\leq m-1$ be the integers of those~$j$ in $\{0, \cdots, m -1\}$ such that $\hW_j$ intersects $\Crit(f)$, and put $n_{D+1}=m.$  By Lemma~\ref{l:maxmeasure},
 for every $i$ in $\{0, 1, \cdots, D\}$  we have
 \begin{equation}\label{e:w2}
 \mu_f(f^{n_{i+1}}(W))=\exp((n_{i+1}-n_{i}-1) h_\tp(f))\mu_f(f^{n_i+1}(W)).
  \end{equation}
 On the other hand, by  our choice of $\tau$ we have that each of $W$ and $f^{n_i}(W)$, $i\in \{1, \cdots, D\}$, contains at most one critical point, so \begin{equation}\label{e:w1}
2^{-1}\exp(n_1 h_\tp(f))\mu_f(W) \leq \mu_f(f^{n_1}(W))\leq \exp(n_1 h_\tp(f))\mu_f(W),
 \end{equation}
 and
 \begin{equation}\label{e:w3} 2^{-1} \exp( h_\tp(f)) \ \mu_f(f^{n_i}(W))\leq \mu_f (f^{n_i+1}(W))\leq \exp( h_\tp(f)) \ \mu_f(f^{n_i}(W)).
 \end{equation}
Combing~(\ref{e:w2}), ~(\ref{e:w1}) and~(\ref{e:w3}), we obtain $$\exp(-m h_\tp(f))\mu_f(f^m(W)) \leq \mu_f(W)\leq 2^D \exp(-m h_\tp(f))\mu_f(f^m(W)),$$ and complete the proof.
\end{proof}

\begin{lemma}\label{orderlength}
There is $\delta_*>0$, $M>1$ and $\tau_*>0$ such that for each $\tau > \tau_*$ the following holds. Let $T$ be a subinterval of $I$ of the length at most $\delta_*$, and let $K\subset J$ be two subintervals of $T$ such that both of connected components of $J\setminus K$ have the length at least $|K|,$ and such that $J$ is $\tau$-well inside in $T$. For each $n\geq 1$ and every pull-back $K_n$ of $K$ by $f^n$ containing a critical point $c$, let~$J_n$ be the pull-back of $J$ by $f^n$ containing~$K_n,$ and let $J_{n-1}$ and $K_{n-1}$ be the pull-back of $J$ and $K$ by $f^{n-1}$ containing $f(K_n)$, respectively. If $f^{n-1}$ maps diffeomorphically a neighborhood of $f(J_n)$ onto $T$ and $$\frac{\max\{|f^n(c)-x|: x\in \partial K\}}{\min\{|f^n(c)-x|: x\in \partial K\}}\leq 2,$$ then for each connected component $W$ of $J_n\setminus K_n$, we have $f^n(W)$ is one of connected components of $J\setminus K$, and $$\frac{|W|}{|K_n|}\leq M\left(\frac{|f^n(W)|}{|K|}\right)^{1/\ell_{\min}}.$$
\end{lemma}
\begin{proof} Fix a sufficiently large number $\tau_*>1,$ and let   $C$ and $\delta_*$ be the constants given by Lemma~\ref{koebe} with $\tau=\tau_*$. Enlarging $\tau_*$ if necessary, we assume that $C\leq 3/2$. Reducing $\delta_*$ if necessary, we also assume that each pull-back of each interval of length at most $\delta_*$ contains at most one critical point.  By the hypothesis, each connected component of $J_{n-1}\setminus K_{n-1}$ is mapped diffeomorphically onto one of  connected component of $J\setminus K$. Moreover, by our choice of $\delta_*$ we know that  each connected component $W$ of $J_n\setminus K_n$ is mapped diffeomorphically onto one of  connected component of  $J_{n-1}\setminus K_{n-1}.$ This proves that $f^n(W)$ is one of connected component of $J\setminus K.$

The desired inequality follows immediately from Koebe distortion and the non-flatness of critical points. In fact, by  Lemma~\ref{koebe}  we have $$ \frac{|f(W)|}{|K_{n-1}|} \leq \frac{3}{2}\frac{|f^n(W)|}{|K|}.$$ On the other hand, by the non-flatness of critical points, there is a constant $C$ depending only on $f$ such that $$\frac{|W|}{|K_n|}\leq C\left(\frac{|f(W)}{|K_{n-1}|}\right)^{1/\ell_c}.$$ This proves the desired inequality with $M=(3/2)^{1/\ell_{\min}}C.$ \end{proof}

\subsection{Topologically Collect-Eckmann condition}
 Let $f: I\to I$
be a multimodal  map and fix $r > 0.$ Recall that given
an integer $n \geq 1,$ the criticality of $f^n$ at a point $x$ of $I$ with respect to $r$ is the number of those $j$
in $\{0, \cdots, n-1\}$ such that the connected component of $f^{(n-j)}(B(f^n(x), r))$ containing
$f^j(x)$ contains a critical point of~$f$ in $\Crit'(f).$  We say that  $f$ satisfies the \emph{Topological Collet-
Eckmann (TCE) condition,} if for some choice of $r > 0$ there are constants $D\geq 1$
and~$\theta$ in $(0, 1),$ such that the following property holds: For each point $x$ in $I$ the
set $G_x$ of all those integers $m\geq 1$ for which the criticality of $f^m$ at $x$ is less than
or equal to $D,$ satisfies $$\liminf_{n\to +\infty}\frac{1}{n} \#(G_x \cap \{1, \cdots, n\}) \geq \theta.$$ Clearly, \emph{every semi-hyperbolic interval map satisfies the TCE condition.} The Topological Collet-Eckmann condition was first introduced in~\cite{NowPrz98}. We will use the following fact that let $f: I\to I$ be a multimodal  interval map in $\sA$ that is topologically exact, then the TCE condition is characterized by each of the following conditions, see for example~\cite[Corollary~C]{Riv1204} for a proof,
\begin{enumerate}[1.]
\item {\bf Exponential Shrinking of Components condition (ESC).} There are constants $\delta > 0$
and $\lambda > 1$ such that for every interval $J$ contained in $I$ that satisfies $|J|\leq \delta,$
the following holds: For every integer $n\geq 1$ and every connected component
$W$ of $f^{-n}(J)$ we have $|W| \leq  \lambda^{-n}.$
\item {\bf Uniform hyperbolicity on periodic orbits.} There is $\lambda>1$ such that for each integer $n\geq 1$ and each repelling periodic point $p$ of period $n$ we have $|D f^n(p)|\geq \lambda^n.$
\end{enumerate}

The following proposition gives another characterization of the TCE condition, see for example~\cite{Riv1204} for a proof; see also~\cite[Theorem~B]{Riv10} for a similar result where $f$ is a rational map.
\begin{lemma}[Remark~6.2, \cite{Riv1204}]\label{p:maximal2tce}
Let $f:I \to I$ be a  multimodal  interval map in $\sA$ that is topologically exact, and let $\mu_f$ be the maximal entropy measure of $f.$  Then $f$ satisfies the TCE condition if and only if
there are constants $r_0 > 0, \alpha > 0$ and $C > 0$ such that for all $x$ in $I$ and $r$ in  $(0, r_0)$
we have $$\mu_f (B(x, r)) \geq  C r^\alpha,$$
\end{lemma}

\subsection{Nice sets and nice couples}
Recall that an open subset~$V$ of~$I$ is \emph{a nice set for~$f$}, if for every integer $n\ge 0$ we have $f^n(\partial V)\cap V=\emptyset$, and
each connected component of~$V$ contains exactly one critical point of~$f$.
In this case, for each~$c$ in~$\Crit(f)$ we denote by~$V^c$ the connected component of~$V$ containing~$c$. A \emph{nice couple for $f$} is a pair of nice sets $(\hV , V )$ such that $\overline{V} \subset \hV$, and such that for every integer $n \ge 1$ the set~$f^n(\partial V)$ is disjoint from~$\hV$.
Moreover, for a  nice nice couple~$(\hV, V)$, we define the \emph{modulus of $(\hV, V)$} as
 $$\md (\hV, V):= \min\{\md (\hV^c, V^c): c\in \Crit(f)\}$$ where
$$\md (\hV^c, V^c):=\sup\{\tau>0: V^c \text{ is $\tau$-well inside } \hV^c\}.$$

\begin{lemma}\label{arblarge}
Assume that $f$ has arbitrarily small
nice couples of arbitrarily large modulus. Then for each recurrent critical point $c_0$
in $\Crit(f)$, $\kappa\in (0, 1)$, $N\geq 2$ and each $r_* > 0$ there is $c\in  \Crit(f)$, $r\in (0, r_*)$ and an
integer $m\geq 1$, such that $f^m(c_0) \in B_I(c, r)$ and such that the pull-back $\hU$ (resp. $U$)
of $B_I(c, r)$ (resp. $B_I(c, \kappa r)$) containing $c_0$ satisfies the following properties.
\begin{enumerate}[1.]
\item $|\hU|<r_*$;
\item The criticality of $f^m$ at $x$ with respect to $\kappa r$ is equal to $N$;
\item $f^m$ maps diffeomorphically each connected component of $\hU\setminus U$  onto one of the connected component of $B_I(c,r)\setminus B_I(c, \kappa r)$.
\end{enumerate}
\end{lemma}
\begin{proof}
See~\cite[Lemma~5]{Riv10} for a proof. There it is proved for rational maps,  but the proof can be adapted to yield the lemma.
\end{proof}

We also use the following lemma. In the case $f$ is a complex rational map, it is~\cite[Proposition~4.2]{PrzRiv07}. The proof applies without changes to the case where~$f$ is an interval map.
\begin{lemma}
\label{l:a priori bounds}
Assume that the map~$f$ satisfies  the TCE condition. Then for every $\delta>0$ and $\tau>0$ there is a nice couple $(\hV, V)$
 of modulus at leat $\tau$
 satisfying $\hV \subset B(\Crit(f), \delta)$.
\end{lemma}

\begin{proof}[Proof of Proposition~\ref{semidoubl}]
Arguing by contradiction, we assume that $f$ has a recurrent critical point~$c_0$. Let $\delta_*>0$, $M>1$ and $\tau_*>0$ be the constants given by Lemma~\ref{orderlength}, and let $\delta_0$ be the constant given by Lemma~\ref{l:pullstable} with $\kappa=\delta_*.$ Reducing $\delta_0$ if necessary we assume $\delta_0<\delta_*.$  Moveover, in view of Lemma~\ref{l:doubleboundedfrombelow} and Lemma~\ref{p:maximal2tce}, we know that~$f$ satisfies the TCE condition. Therefore, by Lemma~\ref{l:a priori bounds} we know that the map~$f$ satisfies the hypothesis of Lemma~\ref{arblarge}.
Let $C$ and $r_0$ be the constants given by Lemma~\ref{measureadj}. Put $$\ell_0:=\sum_{i=0}^{+\infty}\frac{1}{\ell_{\min}^i}, \,\,\, M_0:=M^{\ell_0}, \,\,\, M_1:=\sum_{i=0}^{2M_0} C^i$$ and let $\eps_1>0$ be the constants given by Lemma~\ref{lebsmall2doublsmall} with $M'=4M_1$, and let $N\geq 2$ be the integer such that $\eps_1^{1/\ell_{\min}^N}\leq 2.$ Moreover, let $\eps_2>0$ be the constants given by Lemma~\ref{lebsmall2doublsmall} with $M'=2^{N+1}$. Choose a sufficiently large $\eps_3>0$ so that the constant~$\tau'$ given by Lemma~\ref{improvedkoebe} for~$\tau=\eps_3$ and $N$ is at least~$\tau_*.$

Now let $\hU,U, m, r, c$ be given by Lemma~\ref{arblarge} for $N$ and $c_0$ as above and with
$r_*=\min\{r_0, \delta_0\}$ and $$\kappa=\frac{1}{(1+2\eps_3)(1+2\eps_2+4\eps_2\eps_1)}.$$ By the definition, we have that $\hU$ and $U$ are the connected components
of $f^{-m}(B_I(c, r))$ and $f^{-m}(B_I(c, \kappa r))$ containing $c_0$, respectively. Put $$r_1:=(1+ 2\eps_2) \kappa r \text{ and } r_2:= (1+2\eps_2+2\eps_1\eps_2)\kappa r.$$ Let $\hL_1$ and $\hR_1$ be the left-hand and right-hand connected components of $B_I(c, r_1)\setminus B_I(c, \kappa r)$, respectively. Let $\hL_2$ and  $\hR_2$ be the left-hand and right-hand connected components of $B_I(c, r_2)\setminus B_I(c, r_1)$, respectively. The by the definitions of $r_1$ and $r_2$, and Lemma~\ref{lebsmall2doublsmall}, we have
\begin{equation}\label{biebal1}
\frac{\mu_f(\hL_2)}{\mu_f(\hL_1)}\geq 4M_1\,\,\, \text{ and }\,\,\, \frac{\mu_f(\hR_2)}{\mu_f(\hR_1)}\geq 4M_1,
\end{equation}
and
\begin{equation}\label{biebal2}
 \frac{\min\{\mu_f(\hL_1), \mu_f(\hR_1)\}}{\mu_f(B_I(c, \kappa r))}\geq 2^{N+1}.
 \end{equation}

Let $B$ and $\hB$ be the connected component of $B_I(c, r_1)$ and $B_I(c, r_2)$ by $f^m$ containing~$c_0$, respectively. It follows that $f^m$ maps diffeomorphically each connected component of $\hB\setminus U$ onto one of the connected components of $B_I(c, r_2)\setminus B_I(c, \kappa r).$ In particular, if letting $L_{N}$ be left-hand connected component of $\hB \setminus B$, then by Lemma~\ref{l:maxmeasure} we have
$\mu_f(L_N)=\exp(-m h_\tp(f))\mu_f(f^m(L_N))$, and by Lemmas~\ref{l:maxmeasure} and~\ref{l:pul}
\begin{multline}\label{imagmeas}
\mu_f(B)=\mu_f(B\setminus U)+\mu_f(U)\\\leq 2\exp(-m h_\tp(f))\mu_f(f^m(B\setminus U)\}+2^{N+1}\exp(-m h_\tp(f))\mu_f(B_I(c, \kappa r))
\end{multline}
This, together with~(\ref{biebal2}) and~(\ref{biebal1}), gives us
 \begin{multline}\label{largemes}
 \frac{\mu_f(L_N)}{\mu_f(B)}\geq \frac{\exp(-m h_\tp(f))\mu_f(f^m(L_N))}{3 \exp(-m h_\tp(f)) \mu_f(f^m(B\setminus U))}=\frac{\mu_f(f^m(L_N))}{3 \mu_f(f^m(B\setminus U))}\geq \frac{4M_1}{3}.
 \end{multline}

On the other hand, by Lemma~\ref{orderlength} we know that $$\frac{|L_N|}{|B|}\leq M^{1+\frac{1}{\ell_{\min}}+\cdots+\frac{1}{\ell_{\min}^N}}\left(\frac{|f^m(L_N)|}{|B_I(c, \kappa r)|}\right)^{1/\ell_{\min}^N}\leq M^{\ell_0}\eps_1^{1/\ell_{\min}^N}\leq 2M_0. $$ Therefore, by Lemma~\ref{measureadj} inductively we have $$\mu_f(L_N)\leq (C+C^2+\cdots +C^{2M_0})\mu_f(B)\leq M_1 \mu_f(B).$$ This contradicts the inequality~(\ref{largemes}), and completes the proof of the proposition.
\end{proof}

\end{document}